\documentclass[9pt,shortpaper,twoside,web]{IEEEtran}
\usepackage{cite}

\ifCLASSINFOpdf
  \usepackage[pdftex]{graphicx}
  \usepackage{epstopdf}
  \usepackage{subfigure}
  \graphicspath{{./}{./figures/}}
\else
\fi
\usepackage{xcolor}

\usepackage{amsmath}
\usepackage{amssymb}

\def\R{\mathbb{R}}

\newtheorem{nnassumption}{\bf Assumption}

\newtheorem{nntheorem}{\bf Theorem}
\newenvironment{theorem}{\begin{nntheorem}\it}{\end{nntheorem}}
\newtheorem{nncorollary}{\bf Corollary}

\newtheorem{nndefinition}{\bf Definition}

\newtheorem{nnproposition}{\bf Proposition}

\newtheorem{nnproblem}{\bf Problem}

\newtheorem{nnlemma}{\bf Lemma}

\newtheorem{nnremark}{\bf Remark}
\newenvironment{remark}{\begin{nnremark} \rm }{\hfill \hspace*{1pt}\hfill $\circ$\end{nnremark}}
\newtheorem{nnexample}{\bf Example}

\newenvironment{proof}{{\bf Proof.}}{\hfill \hspace*{1pt}\hfill $\Box$}

\allowdisplaybreaks

\begin{document}
%
\title{Global Output Feedback Stabilization of Semilinear Reaction-Diffusion PDEs}
%
%
%

\author{Hugo~Lhachemi and Christophe~Prieur
\thanks{Hugo Lhachemi is with Universit{\'e} Paris-Saclay, CNRS, CentraleSup{\'e}lec, Laboratoire des signaux et syst{\`e}mes, 91190, Gif-sur-Yvette, France (e-mail: hugo.lhachemi@centralesupelec.fr).
\newline\indent Christophe Prieur is with Universit{\'e} Grenoble Alpes, CNRS, Grenoble-INP, GIPSA-lab, F-38000, Grenoble, France (e-mail: christophe.prieur@gipsa-lab.fr). 

The work of H. Lhachemi has been partially supported by ANR PIA funding: ANR-20-IDEES-0002. The work of C. Prieur has been partially supported by MIAI@Grenoble Alpes (ANR-19-P3IA-0003).
}
}

%
%

\markboth{}%
{Lhachemi \MakeLowercase{\textit{et al.}}}
%



\maketitle

\begin{abstract}
This paper addresses the topic of global output feedback stabilization of semilinear reaction-diffusion PDEs. The semilinearity is assumed to be confined into a sector condition. We consider two different types of actuation configurations, namely: bounded control operator and right Robin boundary control. The measurement is selected as a left Dirichlet trace. The control strategy is finite dimensional and is designed based on a linear version of the plant. We derive a set of sufficient conditions ensuring the global exponential stabilization of the semilinear reaction-diffusion PDE. These conditions are shown to be feasible provided the order of the controller is large enough and the size of the sector condition in which the semilinearity is confined into is small enough.
\end{abstract}


%
\IEEEpeerreviewmaketitle

\section{Introduction}\label{sec: Introduction}

The topic of feedback stabilization of linear reaction-diffusion PDEs has been intensively studied in the literature~\cite{boskovic2001boundary,liu2003boundary,prieur2018boundary} using different approaches such as backstepping~\cite{krstic2008boundary} and spectral reduction methods~\cite{russell1978controllability,coron2004global,coron2006global}. The extension of these approaches to the stabilization of semilinear reaction-diffusion PDEs remains challenging. Among the reported contributions, one can find the study of stability by means of strict Lyapunov functionals~\cite{mazenc2011strict}, observer design~\cite{meurer2013extended,kitsos2020high,zhang2020switching}, and control using quasi-static deformations~\cite{coron2004global}, state-feedback~\cite{karafyllis2021lyapunov} or network control~\cite{selivanov2016distributed,wu2019improved}. 

This paper addresses the topic of global output feedback stabilization of 1-D semilinear reaction-diffusion PDEs with Dirichlet/Neumann/Robin boundary conditions. We assume that the semilinearity is confined into a sector condition. We consider two different configurations for the actuation scheme: bounded control operator and right Robin boundary control. The measurement is selected as the left Dirichlet trace. The adopted output feedback control strategy is composed of a finite-dimensional observer~\cite{curtain1982finite,balas1988finite,harkort2011finite} that leverage control architectures introduced first in \cite{sakawa1983feedback} augmented with a LMI based procedure initially proposed in~\cite{katz2020constructive,katzadelayed}. In this paper, we further adopt the enhanced procedures reported in~
\cite{lhachemi2020finite,lhachemi2021nonlinear} that allow the design in a generic and systematic manner of finite-dimensional observer-based control strategies for general 1-D reaction-diffusion PDEs with Dirichlet/Neumann/Robin boundary control and Dirichlet/Neumann boundary measurements. In particular, these procedures allow to perform the control design directly with the control input instead of its time derivative as classically done in the case of boundary control system, see~\cite[Sec.~3.3]{curtain2012introduction} for further details. This capability to work directly with the control input to perform the control design for general 1-D reaction diffusion PDEs with Dirichlet/Neumann/Robin boundary control and Dirichlet/Neumann boundary measurements is key for the systematic extension of the control design procedure to a number of settings such as predictor-based compensation of arbitrarily long input and output delays~\cite{lhachemi2021predictor,lhachemi2021boundary2}. 

This type of approach has also been shown to be efficient for the local stabilization, with estimation of the domain of attraction, of linear reaction-diffusion PDEs in the presence of a saturation~\cite{lhachemi2021local} (see also \cite{mironchenko2020local} in the context of a state-feedback feedback), as well as the global stabilization of linear-reaction-diffusion PDEs in the presence of a sector nonlinearity in the application of the boundary control~\cite{lhachemi2021nonlinear}. Nevertheless, the possibility to successfully design in a systematic manner finite-dimensional control strategies in the presence of nonlinearities for general reaction-diffusion PDEs has only been demonstrated for linear PDEs with a nonlinearity introduced when applying the control input. In this context, the objective of this work is to demonstrate for the first time that finite-dimensional control strategies can also be designed for achieving the output feedback stabilization of semilinear reaction-diffusion PDEs (i.e., in which the nonlinearity applies to the reaction term). Assuming that the nonlinearity satisfies a sector condition, we derive a set of sufficient LMI conditions ensuring the global output feedback stabilization of the plant. We show that the derived stability conditions are always feasible when selecting the order of the observer large enough and for a size of the sector condition small enough.

The rest of the paper is organized as follows. Notations and preliminary properties are summarized in Section~\ref{sec: notation and properties}. The control design for semilinear reaction-diffusion PDEs in the case of a bounded input operator and left Dirichlet measurement is addressed in Section~\ref{sec: case 2}. Then, Section~\ref{sec: case 3} reports the case of right Robin boundary control and left Dirichlet measurement. Finally, concluding remarks are formulated in Section~\ref{sec: conclusion}.

\section{Notation and properties of Sturm-Liouville operators}\label{sec: notation and properties}

\subsection{Notation}
Real spaces $\R^n$ are equipped with the usual Euclidean norm denoted by $\Vert\cdot\Vert$. The associated induced norms of matrices are also denoted by $\Vert\cdot\Vert$. For any two vectors $X$ and $Y$, $ \mathrm{col} (X,Y)$ represents the vector $[X^\top,Y^\top]^\top$. The space of square integrable functions on $(0,1)$ is denoted by $L^2(0,1)$ and is endowed with the inner product $\langle f , g \rangle = \int_0^1 f(x) g(x) \,\mathrm{d}x$. The associated norm is denoted by $\Vert \cdot \Vert_{L^2}$. For an integer $m \geq 1$, $H^m(0,1)$ stands for the $m$-order Sobolev space and is endowed with its usual norm $\Vert \cdot \Vert_{H^m}$. For any symmetric matrix $P \in\R^{n \times n}$, $P \succeq 0$ (resp. $P \succ 0$) indicates that $P$ is positive semi-definite (resp. positive definite).

\subsection{Properties of Sturm-Liouville operators}

Because reaction-diffusion PDEs are inherently related to Sturm-Liouville operators, we summarize here their key properties that will be used in the sequel. 

Let $\theta_1,\theta_2\in[0,\pi/2]$, $p \in \mathcal{C}^1([0,1])$, and $q \in \mathcal{C}^0([0,1])$ with $p > 0$ and $q \geq 0$. The Sturm-Liouville operator~\cite{renardy2006introduction} is defined by 
$$\mathcal{A} = - (pf')' + q f$$ 
on the domain 
\begin{align*}
D(\mathcal{A}) = \{ f \in H^2(0,1) \,:\, & \cos(\theta_1) f(0) - \sin(\theta_1) f'(0) = 0 ,\,\\ 
& \cos(\theta_2) f(1) + \sin(\theta_2) f'(1) = 0 \}.
\end{align*}
Then the eigenvalues $\lambda_n$, $n \geq 1$, of $\mathcal{A}$ are simple, non-negative, and form an increasing sequence with $\lambda_n \rightarrow + \infty$ as $n \rightarrow + \infty$. The corresponding unit eigenvectors $\phi_n \in L^2(0,1)$ form a Hilbert basis. The domain of the operator $\mathcal{A}$ is equivalently characterized by 
$$D(\mathcal{A}) = \left\{ f \in L^2(0,1) \,:\, \sum_{n\geq 1} \vert \lambda_n \vert ^2 \vert \left< f , \phi_n \right> \vert^2 < +\infty \right\} .$$
Moreover we have $\mathcal{A} f = \sum_{n \geq 1} \lambda_n \left< f , \phi_n \right> \phi_n$ for all $f \in D(\mathcal{A})$. 

Let $p_*,p^*,q^* \in \R$ be such that $0 < p_* \leq p(x) \leq p^*$ and $0 \leq q(x) \leq q^*$ for all $x \in [0,1]$, then it holds that:
\begin{equation}\label{eq: estimation lambda_n}
0 \leq \pi^2 (n-1)^2 p_* \leq \lambda_n \leq \pi^2 n^2 p^* + q^*
\end{equation}
for all $n \geq 1$, see e.g.~\cite{orlov2017general}. Assuming further that $q > 0$, performing an integration by parts and using the continuous embedding $H^1(0,1) \subset L^\infty(0,1)$, we obtain the existence of constants $C_1,C_2 > 0$ such that 
\begin{align}
C_1 \Vert f \Vert_{H^1}^2 \leq 
\sum_{n \geq 1} \lambda_n \left< f , \phi_n \right>^2
= \left< \mathcal{A}f , f \right>
\leq C_2 \Vert f \Vert_{H^1}^2 \label{eq: inner product Af and f}
\end{align}
for all $f \in D(\mathcal{A})$. This implies that $f(0) = \sum_{n \geq 1} \left< f , \phi_n \right> \phi_n(0)$ and $f'(0) = \sum_{n \geq 1} \left< f , \phi_n \right> \phi_n'(0)$ hold for all $f \in D(\mathcal{A})$. Finally, if we further assume that $p \in \mathcal{C}^2([0,1])$, we have for any $x \in [0,1]$ that $\phi_n(x) = O(1)$ and $\phi_n'(x) = O(\sqrt{\lambda_n})$ as $n \rightarrow +\infty$, see e.g.~\cite{orlov2017general}. 

For an arbitrarily given integer $N \geq 1$, we define $\mathcal{R}_N f = \sum_{n \geq N+1} \left< f , \phi_n \right> \phi_n$.

\section{Distributed command and Dirichlet boundary measurement}\label{sec: case 2}

\subsection{Problem setting and spectral reduction}

We study first the stabilization problem of the reaction-diffusion PDE described by
\begin{subequations}\label{eq: RD system 1}
\begin{align}
& z_t(t,x) = \left( p(x) z_x(t,x) \right)_x - \tilde{q}_0(x) z(t,x) + f(t,x,z(t,x)) \nonumber \\
& \phantom{z_t(t,x) = }\; + b(x) u(t) \label{eq: RD system 1 - 1} \\
&\cos(\theta_1) z(t,0) - \sin(\theta_1) z_x(t,0) = 0 \label{eq: RD system 1 - 2} \\
&\cos(\theta_2) z(t,1) + \sin(\theta_2) z_x(t,1) = 0 \label{eq: RD system 1 - 3} \\
& z(0,x) = z_0(x) . \label{eq: RD system 1 - 4}
\end{align}
\end{subequations}
for $t > 0$ and $x \in(0,1)$. Here we have $\theta_1 \in (0,\pi/2]$, $\theta_2 \in [0,\pi/2]$, $p \in \mathcal{C}^2([0,1])$ with $p > 0$, and $\tilde{q}_0 \in \mathcal{C}^0([0,1])$. The distributed control input is $u(t) \in\R$ and acts on the system via the shape function $b \in L^2(0,1)$. The state of the reaction-diffusion PDE is $z(t,\cdot) \in L^2(0,1)$ while $z_0 \in L^2(0,1)$ is the initial condition. The function $f : \R_+ \times[0,1]\times\R \rightarrow \R$ is assumed to be globally Lipchitz continuous in $z$, uniformly in $(t,x)$, so that $f(\cdot,\cdot,0)=0$. Let $\tilde{q}_f \in \mathcal{C}^0([0,1])$ and $k_f > 0$ be such that 
\begin{equation}\label{eq: sector condition}
\vert f(t,x,z) - \tilde{q}_f(x) z \vert \leq k_f \vert z \vert , \quad \forall t \geq 0, \quad \forall x\in[0,1], \quad \forall z \in\R .
\end{equation}
Hence (\ref{eq: RD system 1 - 1}) can be written as 
\begin{align*}
& z_t(t,x) = \left( p(x) z_x(t,x) \right)_x - \tilde{q}(x) z(t,x) + g(t,x,z(t,x)) \\
& \phantom{z_t(t,x) = }\; + b(x) u(t)
\end{align*}
where $\tilde{q} = \tilde{q}_0 - \tilde{q}_f$ and $g(t,x,z) = f(t,x,z) - \tilde{q}_f(x) z$ with 
\begin{equation}\label{eq: sector condition bis}
\vert g(t,x,z) \vert \leq k_f \vert z \vert , \quad \forall t \geq 0 , \quad \forall x\in[0,1], \quad \forall z \in\R .
\end{equation}
The system output takes the form of the left Dirichlet trace:
\begin{equation}\label{eq: Dirichlet measurement operator}
y(t) = z(t,0) .
\end{equation}

In perspective of control design, recalling that $\tilde{q} = \tilde{q}_0 - \tilde{q}_f$, and without loss of generality, we pick a function $q \in \mathcal{C}^0([0,1])$ and a constant $q_c \in \R$ such that
\begin{equation}\label{eq: decomposition of the reaction term}
\tilde{q} = q - q_c , \quad q > 0 .
\end{equation} 
The projection of (\ref{eq: RD system 1}) into the Hilbert basis $(\phi_n)_{n \geq 1}$ gives 
\begin{equation}\label{eq: RD system 1 spectral reduction 1}
\dot{z}_n (t)= (-\lambda_n+q_c) z_n(t) + b_n u(t) + g_n(t)
\end{equation}
where $z_n(t) = \left< z(t,\cdot) , \phi_n \right>$, $b_n = \left< b , \phi_n \right>$, and $g_n(t) = \left< g(t,\cdot,z(t,\cdot)) , \phi_n \right>$. Moreover, considering classical solutions, the system output (\ref{eq: Dirichlet measurement operator}) is expressed as the series expansion:
\begin{equation}\label{eq: RD system 1 spectral reduction 3}
y(t) = \sum_{n \geq 1} \phi_n(0) z_n(t) . 
\end{equation}

\subsection{Control design and truncated model for stability analysis}\label{subsec: control design 1}

Let $\delta > 0$ be the desired exponential decay rate for the closed-loop system trajectories. Let an integer $N_0 \geq 1$ be such that $-\lambda_n + q_c < - \delta < 0$ for all $n \geq N_0 + 1$. We introduce an arbitrary integer $N \geq N_0 + 1$ that will be specified later. The control strategy takes the form of the following dynamics: 
\begin{subequations}\label{eq: controller 2}
\begin{align}
\dot{\hat{z}}_n(t) & = (-\lambda_n+q_c) \hat{z}_n(t) + b_n u(t) \nonumber \\
& \phantom{=}\; - l_n \left\{ \sum_{k = 1}^N \phi_k(0) \hat{z}_k(t) - y(t) \right\} ,\quad 1 \leq n \leq N_0 \label{eq: controller 2 - 1} \\
\dot{\hat{z}}_n(t) & = (-\lambda_n+q_c) \hat{z}_n(t) + b_n u(t) ,\quad N_0+1 \leq n \leq N \label{eq: controller 2 - 2} \\
u(t) & = \sum_{n=1}^{N_0} k_n \hat{z}_n(t) \label{eq: controller 2 - 3}
\end{align}
\end{subequations}
where $l_n,k_n \in\R$ are the observer and feedback gains, respectively. 

In preparation for stability analysis, we need to build a truncated model capturing the $N$ first modes of the PDE (\ref{eq: RD system 1}) as well as the dynamics of the controller (\ref{eq: controller 2}). To do so, let us define the error of observation $e_n = z_n - \hat{z}_n$, the scaled error of observation $\tilde{e}_n = \sqrt{\lambda_n} e_n$, and the vectors $\hat{Z}^{N_0} = \begin{bmatrix} \hat{z}_1 & \ldots & \hat{z}_{N_0} \end{bmatrix}^\top$, $E^{N_0} = \begin{bmatrix} e_1 & \ldots & e_{N_0} \end{bmatrix}^\top$, $\hat{Z}^{N-N_0} = \begin{bmatrix} \hat{z}_{N_0+1} & \ldots & \hat{z}_{N} \end{bmatrix}^\top$, $\tilde{E}^{N-N_0} = \begin{bmatrix} \tilde{e}_{N_0 +1} & \ldots & \tilde{e}_{N} \end{bmatrix}^\top$, $R_1 = \begin{bmatrix} g_1 & \ldots & g_{N_0} \end{bmatrix}^\top$, $\tilde{R}_2 = \begin{bmatrix} \sqrt{\lambda_{N_0+1}} g_{N_0 +1} & \ldots & \sqrt{\lambda_{N}} g_{N} \end{bmatrix}^\top$, and $R = \mathrm{col}(R_1,\tilde{R}_2)$. We also introduce the matrices defined by $A_0 = \mathrm{diag}(-\lambda_1+q_c,\ldots,-\lambda_{N_0}+q_c)$, $A_1 = \mathrm{diag}(-\lambda_{N_0+1}+q_c,\ldots,-\lambda_N+q_c)$, $B_0 = \begin{bmatrix} b_1 & \ldots & b_{N_0} \end{bmatrix}^\top$, $B_1 = \begin{bmatrix} b_{N_0+1} & \ldots & b_N \end{bmatrix}^\top$, $C_0 = \begin{bmatrix} \phi_1(0) & \ldots & \phi_{N_0}(0) \end{bmatrix}$, $\tilde{C}_1 = \begin{bmatrix} \frac{\phi_{N_0+1}(0)}{\sqrt{\lambda_{N_0+1}}} & \ldots & \frac{\phi_{N}(0)}{\sqrt{\lambda_{N}}} \end{bmatrix}$, $K = \begin{bmatrix} k_1 & \ldots & k_{N_0} \end{bmatrix}$, and $L = \begin{bmatrix} l_1 & \ldots & l_{N_0} \end{bmatrix}^\top$. Therefore, we obtain from (\ref{eq: RD system 1 spectral reduction 1}-\ref{eq: RD system 1 spectral reduction 3}) and (\ref{eq: controller 2}) that
\begin{subequations}\label{eq: RD system 1 - dynamics truncated model - 5eq}
\begin{align}
u & = K \hat{Z}^{N_0} \\
\dot{\hat{Z}}^{N_0} & = (A_0 + B_0 K) \hat{Z}^{N_0} + LC_0 E^{N_0} + L \tilde{C}_1 \tilde{E}^{N-N_0} + L \zeta \\
\dot{E}^{N_0} & = (A_0 - LC_0) E^{N_0} - L \tilde{C}_1 \tilde{E}^{N-N_0} - L \zeta + R_1 \\
\dot{\hat{Z}}^{N-N_0} & = A_1 \hat{Z}^{N-N_0} + B_1 K \hat{Z}^{N_0} \\
\dot{\tilde{E}}^{N-N_0} & = A_1 \tilde{E}^{N-N_0} + \tilde{R}_2 .
\end{align}
\end{subequations}
with $\zeta = \sum_{n \geq N+1} \phi_n(0) z_n$. Introducing the state vector
\begin{equation}\label{eq: RD system 2 - state truncated model}
X = \mathrm{col}\left( \hat{Z}^{N_0} , E^{N_0} , \hat{Z}^{N-N_0} , \tilde{E}^{N-N_0} \right) ,
\end{equation}
we infer that the following truncated dynamics hold
\begin{equation}\label{eq: RD system 2 dynamics truncated model}
\dot{X} = F X + \mathcal{L} \zeta + G R
\end{equation}
where
\begin{equation*}
F = \begin{bmatrix}
A_0 + B_0 K & LC_0 & 0 & L\tilde{C}_1 \\
0 & A_0 - L C_0 & 0 & - L\tilde{C}_1 \\
B_1 K & 0 & A_1 & 0 \\
0 & 0 & 0 & A_1
\end{bmatrix}
, \quad 
G = \begin{bmatrix}
0 & 0 \\ I & 0 \\ 0 & 0 \\ 0 & I
\end{bmatrix} 
\end{equation*}
and $\mathcal{L} = \mathrm{col}(L,-L,0,0)$. Defining $\tilde{K} = \begin{bmatrix} K & 0 & 0 & 0 \end{bmatrix}$, we also have that 
\begin{equation}\label{eq: u in function of X}
u = \tilde{K} X .
\end{equation}
We finally define the matrices 
\begin{equation*}
\Lambda = \mathrm{diag}(\lambda_{N_0+1},\ldots,\lambda_{N})
, \quad
\Omega = \begin{bmatrix}
I & I & 0 & 0 \\ I & I & 0 & 0 \\ 0 & 0 & I & \Lambda^{-1/2} \\ 0 & 0 & \Lambda^{-1/2} & \Lambda^{-1}
\end{bmatrix}
\end{equation*}
and $\tilde{\Lambda} = \mathrm{diag}(I,\Lambda)$. In particular, we have that $\Omega \preceq 2 \max(1,1/\lambda_{N_0+1}) I$ and $\tilde{\Lambda}^{-1} \succeq \min(1,1/\lambda_N) I$.

\subsection{Stability result}

\begin{theorem}\label{thm2}
Let $\theta_1 \in (0,\pi/2]$, $\theta_2 \in [0,\pi/2]$, $p \in\mathcal{C}^2([0,1])$ with $p > 0$, $\tilde{q}_0 \in\mathcal{C}^0([0,1])$, and $b \in L^2(0,1)$. Let $f : \R_+\times[0,1]\times\R \rightarrow \R$ be globally Lipchitz continuous in $z$, uniformly in $(t,x)$, so that $f(\cdot,\cdot,0)=0$. Let $\tilde{q}_f \in \mathcal{C}^0([0,1])$ and $k_f > 0$ be so that (\ref{eq: sector condition}) holds. Let $q \in\mathcal{C}^0([0,1])$ and $q_c \in\R$ be such that (\ref{eq: decomposition of the reaction term}) holds. Let $\delta > 0$ and $N_0 \geq 1$ be such that $-\lambda_n + q_c < - \delta$ for all $n \geq N_0 + 1$. Assume that\footnote{This implies that $(A_0,B_0)$ satisfies the Kalman condition. Note that $(A_0,C_0)$ satisfies the Kalman condition by arguments from~\cite{lhachemi2021nonlinear}.} $b_n \neq 0$ for all $1 \leq n \leq N_0$. Let $K\in\R^{1 \times N_0}$ and $L\in\R^{N_0}$ be such that $A_0 + B_0 K$ and $A_0 - L C_0$ are Hurwitz with eigenvalues that have a real part strictly less than $-\delta<0$. For a given $N \geq N_0 +1$, assume that there exist a symmetric positive definite $P\in\R^{2N \times 2N}$, positive real numbers $\alpha_1,\alpha_2 > 1$ and $\beta,\gamma > 0$ such that 
\begin{equation}\label{eq: thm2 constraints}
\Theta_1 \preceq 0 , \quad
\Theta_2 \leq 0 
\end{equation}
where
\begin{align*}
\Theta_1 & = \begin{bmatrix}
\Theta_{1,1,1} & P \mathcal{L} & P G  \\
\mathcal{L}^\top P & -\beta & 0 \\
G^\top P & 0 & - \alpha_2 \gamma \tilde{\Lambda}^{-1}
\end{bmatrix} \\
\Theta_{1,1,1} & = F^\top P + P F + 2 \delta P + \alpha_1\gamma \Vert \mathcal{R}_N b \Vert_{L^2}^2 \tilde{K}^\top \tilde{K} + \alpha_2 \gamma k_f^2 \Omega \\
\Theta_2 & = 2 \gamma \left\{ - \left[ 1 - \frac{1}{2} \left( \frac{1}{\alpha_1} + \frac{1}{\alpha_2} \right) \right] \lambda_{N+1} + q_c + \delta \right\} \\
& \phantom{=}\; + \beta M_\phi + \frac{\alpha_2 \gamma k_f^2}{\lambda_{N+1}} 
\end{align*}
with $M_\phi = \sum_{n \geq N+1} \frac{\phi_n(0)^2}{\lambda_n}$. Then, considering the closed-loop system composed of the plant (\ref{eq: RD system 1}) with the system output (\ref{eq: Dirichlet measurement operator}) and the controller (\ref{eq: controller 2}), there exists $M > 0$ such that for any initial conditions $z_0 \in H^2(0,1)$ and $\hat{z}_n(0) \in\R$ so that $\cos(\theta_1) z_0(0) - \sin(\theta_1) z'_0(0) = 0$ and $\cos(\theta_2) z_0(1) + \sin(\theta_2) z'_0(1) = 0$, the system trajectory satisfies 
$$\Vert z(t,\cdot)\Vert_{H^1}^2 + \sum_{n=1}^N \hat{z}_n(t)^2 \leq M e^{-2\kappa t} \left( \Vert z_0 \Vert_{H^1}^2  + \sum_{n=1}^N \hat{z}_n(0)^2 \right)$$ 
for all $t \geq 0$. Moreover, when selecting $N$ to be sufficiently large, there exists $k_f > 0$ (small enough) so that the constraints (\ref{eq: thm2 constraints}) are feasible.
\end{theorem}

\begin{proof}
Consider the Lyapunov functional defined for $X\in\R^{2N}$ and $z \in D(\mathcal{A})$ by (see, e.g., \cite{coron2004global})
$$V(X,z) = X^\top P X + \gamma \sum_{n \geq N+1} \lambda_n z_n^2 .$$
The computation of the time derivative of $V$ along the system trajectories (\ref{eq: RD system 1 spectral reduction 1}) and (\ref{eq: RD system 2 dynamics truncated model}) reads
\begin{align*}
& \dot{V} + 2 \delta V \\
& = X^\top \{ F^\top P + P F + 2 \delta P \} X + 2 X^\top P \mathcal{L} \zeta + 2 X^\top P G R \\
& \phantom{=}\; + 2\gamma\sum_{n \geq N+1} \lambda_n (-\lambda_n+q_c+\delta) z_n^2 + 2\gamma\sum_{n \geq N+1} \lambda_n (b_n u + g_n) z_n .
\end{align*} 
We define $\tilde{X} = \mathrm{col}(X,\zeta,R)$. Noting that $\zeta^2 \leq M_\phi \sum_{n \geq N+1} \lambda_n z_n^2$ and using Young's inequality and (\ref{eq: u in function of X}), we deduce that
\begin{align*}
& \dot{V} + 2 \delta V \leq \\
& \tilde{X}^\top 
\begin{bmatrix}
F^\top P + P F + 2 \delta P + \alpha_1\gamma \Vert \mathcal{R}_N b \Vert_{L^2}^2 \tilde{K}^\top \tilde{K} & P\mathcal{L} & P G \\
\mathcal{L}^\top P & -\beta & 0 \\
G^\top P & 0 & 0
\end{bmatrix}  
\tilde{X} \\
& + 2\gamma \sum_{n \geq N+1} \lambda_n \left\{ -\lambda_n + q_c + \delta + \frac{1}{2} \left( \frac{1}{\alpha_1} + \frac{1}{\alpha_2} \right) \lambda_n \right\} z_n^2 \\
& + \beta M_\phi \sum_{n \geq N+1} \lambda_n z_n^2 + \alpha_2 \gamma \Vert \mathcal{R}_N g(t,\cdot,z) \Vert_{L^2}^2 
\end{align*}
for any $\alpha_1,\alpha_2 > 0$. Owing to the sector condition (\ref{eq: sector condition bis}), we infer that
\begin{align}
\Vert \mathcal{R}_N g(t,\cdot,z) \Vert_{L^2}^2
& = \Vert g(t,\cdot,z) \Vert_{L^2}^2 - \sum_{n=1}^N g_n^2 \nonumber \\
& \leq k_f^2 \Vert z \Vert_{L^2}^2 - \Vert R_1 \Vert^2 \nonumber - \sum_{n=N_0+1}^N \frac{1}{\lambda_n} \left( \sqrt{\lambda_n} g_n \right)^2 \\
& \leq k_f^2 \sum_{n=1}^N z_n^2 + k_f^2 \sum_{n \geq N+1} z_n^2 - R^\top \tilde{\Lambda}^{-1} R . \label{eq: norm residue of semilinearity}
\end{align} 
Recalling that $e_n = z_n - \hat{z}_n$ and $\tilde{e}_n = \sqrt{\lambda}_n e_n$, we deduce that $\sum_{n=1}^{N_0} z_n^2 = \hat{z}_n^2 + 2 \hat{z}_n e_n + e_n^2$ and $\sum_{n=N_{0}+1}^{N} z_n^2 = \hat{z}_n^2 + \frac{2}{\sqrt{\lambda_n}} \hat{z}_n \tilde{e}_n + \frac{1}{\lambda_n} \tilde{e}_n^2$, hence $\sum_{n=1}^{N} z_n^2 = X^\top \Omega X$. This implies that

\begin{equation*}
\dot{V} + 2 \delta V 
\leq \tilde{X}^\top \Theta_1 \tilde{X} + \sum_{n \geq N+1} \lambda_n \Gamma_n z_n^2 
\end{equation*}
where $\Gamma_n = 2\gamma \left\{ - \left[ 1 - \frac{1}{2} \left( \frac{1}{\alpha_1} + \frac{1}{\alpha_2} \right) \right] \lambda_n + q_c + \delta \right\} + \beta M_\phi + \frac{\alpha_2 \gamma k_f^2}{\lambda_n} \leq \Theta_2$ for all $n \geq N+1$ where we have used that $1 - \frac{1}{2} \left( \frac{1}{\alpha_1} + \frac{1}{\alpha_2} \right) > 0$ because $\alpha_1,\alpha_2 > 1$. Owing to (\ref{eq: thm2 constraints}), we infer that $\dot{V} + 2 \delta V \leq 0$. Using the definition of $V$ along with (\ref{eq: inner product Af and f}), we deduce that the claimed exponential stability estimate holds true.

It remains to show that the constraints (\ref{eq: thm2 constraints}) are feasible when $N$ is large enough and $k_f > 0$ is small enough. First, we observe that (i) $A_0 + \mathfrak{B}_0 K + \delta I$ and $A_0 - L C_0 + \delta I$ are Hurwitz; (ii) $\Vert e^{(A_1+\delta I) t} \Vert \leq e^{-\kappa_0 t}$ for all $t \geq 0$ with the positive constant $\kappa_0 = \lambda_{N_0 +1} - q_c - \delta > 0$ independent of $N$; (iii) $\Vert L \tilde{C}_1 \Vert \leq \Vert L \Vert \Vert \tilde{C}_1 \Vert$ and $\Vert B_1 K \Vert \leq \Vert b \Vert_{L^2} \Vert K \Vert$ where $\Vert L \Vert$ and $\Vert K \Vert$ are independent of $N$ while $\Vert \tilde{C}_1 \Vert = O(1)$ a $N \rightarrow + \infty$. Hence, the application of the Lemma in appendix of \cite{lhachemi2020finite} to the matrix $F+\delta I$ shows that the unique solution $P \succ 0$ to the Lyapunov equation $F^\top P + P F + 2 \delta P = -I$ is such that $\Vert P \Vert = O(1)$ as $N \rightarrow + \infty$. We fix $\alpha_1>1$ and $\gamma > 0$ arbitrarily. We also fix $\beta=N$, $\alpha_2 = N^3$, and $k_f = 1/N^2$. Hence it can be seen that $\Theta_2 \rightarrow - \infty$ as $N \rightarrow + \infty$. Moreover, since $\Vert \tilde{K} \Vert = \Vert K \Vert$, $\Vert \tilde{L} \Vert = \sqrt{2} \Vert L \Vert$, and $\Vert G \Vert= 1$ are constants independent of $N$, $\Vert P \Vert = O(1)$ as $N \rightarrow + \infty$, and $\Omega \preceq 2 \max(1,1/\lambda_{N_0+1}) I$ and $- \tilde{\Lambda}^{-1} \preceq - \min(1,1/\lambda_N) I$, the Schur complement shows that $\Theta_1 \preceq 0$ for $N$ selected to be large enough. This completes the proof.
\end{proof}

\begin{remark}
For a fixed order $N \geq N_0+1$ of the observer and for a given value of $k_f > 0$, the constraints (\ref{eq: thm2 constraints}) are nonlinear w.r.t. the decision variables $P \succ 0$, $\alpha_1,\alpha_2>1$, and $\gamma,\beta>0$. In order to obtain a LMI formulation of the constraints, we first fix arbitrarily the value of the decision variable $\gamma > 0$. Following the last part of the proof of Theorem~\ref{thm2}, the obtained constraints remain feasible for $N$ large enough and $k_f>0$ small enough. Now $\Theta_1 \preceq 0$ takes the form of a LMI w.r.t. the decision variables $P,\alpha_1,\alpha_2,\beta$ while, using Schur complement, $\Theta_2 \leq 0$ is equivalent to the LMI formulation:
\begin{equation*}
\begin{bmatrix}
\mu & \sqrt{\gamma \lambda_{N+1}} & \sqrt{\gamma \lambda_{N+1}} \\
\sqrt{\gamma \lambda_{N+1}} & -\alpha_1 & 0 \\
\sqrt{\gamma \lambda_{N+1}} & 0 & -\alpha_2
\end{bmatrix}
\preceq 0
\end{equation*}  
with $\mu = 2\gamma \{ -\lambda_{N+1} + q_c + \delta \} + \beta M_\phi + \alpha_2 \gamma k_f^2 / \lambda_{N+1}$. A similar remark applies to the constraints of Theorem~\ref{thm3}.
\end{remark}

\section{Robin boundary control and Dirichlet boundary measurement}\label{sec: case 3}

\subsection{Problem setting and spectral reduction}

We now consider the boundary stabilization of the reaction-diffusion PDE described by
\begin{subequations}\label{eq: RD system 2}
\begin{align}
& z_t(t,x) = \left( p(x) z_x(t,x) \right)_x - \tilde{q}_0 (x) z(t,x) + f(t,x,z(t,x)) \label{eq: RD system 2 - 1} \\
&\cos(\theta_1) z(t,0) - \sin(\theta_1) z_x(t,0) = 0 \label{eq: RD system 2 - 2} \\
&\cos(\theta_2) z(t,1) + \sin(\theta_2) z_x(t,1) = u(t) \label{eq: RD system 2 - 3} \\
& z(0,x) = z_0(x) . \label{eq: RD system 2 - 4}
\end{align}
\end{subequations}
for $t > 0$ and $x \in(0,1)$ where $p \in\mathcal{C}^2([0,1])$ and $\theta_1 \in (0,\pi/2]$. The boundary measurement is selected as the left Dirichlet trace (\ref{eq: Dirichlet measurement operator}). We still assume that there exist $\tilde{q}_f \in \mathcal{C}^0([0,1])$ and $k_f > 0$ so that (\ref{eq: sector condition}) holds. Hence we define $\tilde{q} = \tilde{q}_0 - \tilde{q}_f$ and $g(t,x,z) = f(t,x,z) - \tilde{q}_f(x) z$.

Due to the boundary control $u$, the PDE (\ref{eq: RD system 2}) is non-homogeneous. Hence we introduce the change of variable
\begin{equation}\label{eq: change of variable}
w(t,x) = z(t,x) - \frac{x^2}{\cos\theta_2 + 2\sin\theta_2} u(t)
\end{equation}
so that, introducing $v = \dot{u}$, the PDE (\ref{eq: RD system 2}) is equivalently rewritten under the following homogeneous representation: 
\begin{subequations}\label{eq: RD system 2 homogeneous}
\begin{align}
& \dot{u}(t) = v(t) \label{eq: RD system 2 homogeneous - 0} \\
& w_t(t,x) = \left( p(x) w_x(t,x) \right)_x - \tilde{q}(x) w(t,x) + g(t,x,z(t,x)) \label{eq: RD system 2 homogeneous - 1} \\
& \phantom{w_t(t,x) =}\; + a(x) u(t) + b(x) v(t) \nonumber \\
& \cos(\theta_1) w(t,0) - \sin(\theta_1) w_x(t,0) = 0 \label{eq: RD system 2 homogeneous - 2} \\
& \cos(\theta_2) w(t,1) + \sin(\theta_2) w_x(t,1) = 0 \label{eq: RD system 2 homogeneous - 3} \\
& w(0,x) = w_0(x) . \label{eq: RD system 2 homogeneous - 4}
\end{align}
\end{subequations}
where $a(x) = \frac{1}{\cos\theta_2+2\sin\theta_2} \{ 2p(x) + 2xp'(x) - x^2 \tilde{q}(x) \}$, $b(x) = -\frac{x^2}{\cos\theta_2+2\sin\theta_2}$, and $w_0(x) = z_0(x) - \frac{x^2}{\cos\theta_2+2\sin\theta_2} u(0)$. Note that we kept in (\ref{eq: RD system 2 homogeneous - 1}) the nonlinear term $g(t,x,z(t,x))$ expressed in original coordinate $z$. We introduce $q \in \mathcal{C}^0([0,1])$ and $q_c \in \R$ so that (\ref{eq: decomposition of the reaction term}) holds. Finally, we define the coefficients of projection $z_n(t) = \left< z(t,\cdot) , \phi_n \right>$, $w_n(t) = \left< w(t,\cdot) , \phi_n \right>$, $a_n = \left< a , \phi_n \right>$, $b_n = \left< b , \phi_n \right>$, and $g_n(t) = \left< g(t,\cdot,z(t,\cdot)) , \phi_n \right>$. From (\ref{eq: change of variable}) we deduce that
\begin{equation}\label{eq: link z_n and w_n}
w_n(t) = z_n(t) + b_n u(t), \quad n \geq 1 .
\end{equation}
Moreover, the projection of (\ref{eq: RD system 2 homogeneous}) into the Hilbert basis $(\phi_n)_{n \geq 1}$ gives
\begin{subequations}\label{eq: dynamics w_n}
\begin{align}
\dot{u}(t) & = v(t) \\
\dot{w}_n(t) & = (-\lambda_n + q_c) w_n(t) + a_n u(t) + b_n v(t) + g_n(t) 
\end{align}
\end{subequations}
with $w(t,\cdot) = \sum_{n \geq 1} w_n(t) \phi_n$. The convergence of the latter series holds in $L^2$ norm for mild solutions and in $H^2$ norm for classical solutions. Inserting (\ref{eq: link z_n and w_n}) into (\ref{eq: dynamics w_n}), the projection of (\ref{eq: RD system 2}) gives
\begin{equation}\label{eq: dynamics z_n}
\dot{z}_n(t) = (-\lambda_n + q_c) z_n(t) + \beta_n u(t) + g_n(t)
\end{equation}
where $\beta_n = a_n + (-\lambda_n+q_c)b_n = p(1) \{ - \cos(\theta_2) \phi_n'(1) + \sin(\theta_2) \phi_n(1) \} = O(\sqrt{\lambda_n})$. In particular we have $z(t,\cdot) = \sum_{n \geq 1} z_n(t) \phi_n$ with convergence of the series in $L^2$ norm. Finally, considering classical solutions, the system output (\ref{eq: Dirichlet measurement operator}) is expressed by the series expansion
\begin{equation}\label{eq: RD system 3 spectral reduction 3}
y(t) = z(t,0) = w(t,0) = \sum_{n \geq 1} \phi_n(0) w_n(t) . 
\end{equation}

\subsection{Control design and truncated model for stability analysis}

Let $\delta > 0$ be the desired exponential decay rate for the closed-loop system trajectories. Let an integer $N_0 \geq 1$ be such that $-\lambda_n + q_c < - \delta < 0$ for all $n \geq N_0 + 1$. For an arbitrary integer $N \geq N_0 + 1$, the control strategy reads: 
\begin{subequations}\label{eq: controller 3}
\begin{align}
\hat{w}_n(t) & = \hat{z}_n(t) + b_n u(t) \\
\dot{\hat{z}}_n(t) & = (-\lambda_n+q_c) \hat{z}_n(t) + \beta_n u(t) \nonumber \\
& \phantom{=}\; - l_n \left\{ \sum_{k = 1}^N \phi_k(0) \hat{w}_k(t) - y(t) \right\} ,\quad 1 \leq n \leq N_0 \label{eq: controller 3 - 1} \\
\dot{\hat{z}}_n(t) & = (-\lambda_n+q_c) \hat{z}_n(t) + \beta_n u(t) ,\quad N_0+1 \leq n \leq N \label{eq: controller 3 - 2} \\
u(t) & = \sum_{n=1}^{N_0} k_n \hat{z}_n(t) \label{eq: controller 3 - 3}
\end{align}
\end{subequations}
where $l_n,k_n \in\R$ are the observer and feedback gains, respectively.

To build the truncated model, we adopt the notations of Subsection~\ref{subsec: control design 1} except that we introduce the notations $\tilde{z}_n = \hat{z}_n/\lambda_n$ along with the vector $\tilde{Z}^{N-N_0} = \begin{bmatrix} \tilde{z}_{N_0+1} & \ldots & \tilde{z}_{N} \end{bmatrix}^\top$ and the matrices $\mathfrak{B}_0 = \begin{bmatrix} \beta_1 & \ldots & \beta_{N_0} \end{bmatrix}^\top$ and $\tilde{\mathfrak{B}}_1 = \begin{bmatrix} \frac{\beta_{N_0 +1}}{\lambda_{N_0 +1}} & \ldots & \frac{\beta_{N}}{\lambda_{N}} \end{bmatrix}^\top$. Hence, introducing the state vector
\begin{equation}\label{eq: RD system 3 - state truncated model}
X = \mathrm{col}\left( \hat{Z}^{N_0} , E^{N_0} , \tilde{Z}^{N-N_0} , \tilde{E}^{N-N_0} \right) ,
\end{equation}
we deduce similarly to the developments of Subsection~\ref{subsec: control design 1} that the following truncated dynamics hold
\begin{equation}\label{eq: RD system 3 dynamics truncated model}
\dot{X} = F X + \mathcal{L} \zeta + G R
\end{equation}
where
\begin{equation*}
F = \begin{bmatrix}
A_0 + \mathfrak{B}_0 K & LC_0 & 0 & L\tilde{C}_1 \\
0 & A_0 - L C_0 & 0 & - L\tilde{C}_1 \\
\tilde{\mathfrak{B}}_1 K & 0 & A_1 & 0 \\
0 & 0 & 0 & A_1
\end{bmatrix}
, \quad 
G = \begin{bmatrix}
0 & 0 \\ I & 0 \\ 0 & 0 \\ 0 & I
\end{bmatrix} 
\end{equation*}
and $\mathcal{L} = \mathrm{col}(L,-L,0,0)$. We define the matrices 
\begin{equation*}
\Lambda = \mathrm{diag}(\lambda_{N_0+1},\ldots,\lambda_{N})
, \quad
\Omega = \begin{bmatrix}
I & I & 0 & 0 \\ I & I & 0 & 0 \\ 0 & 0 & \Lambda^2 & \Lambda^{1/2} \\ 0 & 0 & \Lambda^{1/2} & \Lambda^{-1}
\end{bmatrix}
\end{equation*}
and $\tilde{\Lambda} = \mathrm{diag}(I,\Lambda)$ which are such that $\Omega \preceq 2 \max(1,1/\lambda_{N_0+1},\lambda_N^2) I$ and $\tilde{\Lambda}^{-1} \succeq \min(1,1/\lambda_N) I$. Finally, introducing $\tilde{X} = \mathrm{col}(X,\zeta,R)$, we have
\begin{equation}\label{eq: expression of v in function of tilde_X}
v = \dot{u} = K \dot{\hat{Z}}^{N_0} = E \tilde{X}
\end{equation}
where $E = K \begin{bmatrix} A_0 + \mathfrak{B}_0 K & LC_0 & 0 & L\tilde{C}_1 & L & 0 \end{bmatrix}$.

\subsection{Stability result}

\begin{theorem}\label{thm3}
Let $\theta_1 \in (0,\pi/2]$, $\theta_2 \in [0,\pi/2]$, $p \in\mathcal{C}^2([0,1])$ with $p > 0$, and $\tilde{q}_0 \in\mathcal{C}^0([0,1])$. Let $f : \R_+\times[0,1]\times\R \rightarrow \R$ be globally Lipchitz continuous in $z$, uniformly in $(t,x)$, so that $f(\cdot,\cdot,0)=0$. Let $\tilde{q}_f \in \mathcal{C}^0([0,1])$ and $k_f > 0$ be so that (\ref{eq: sector condition}) holds. Let $q \in\mathcal{C}^0([0,1])$ and $q_c \in\R$ be such that (\ref{eq: decomposition of the reaction term}) holds. Let $\delta > 0$ and $N_0 \geq 1$ be such that $-\lambda_n + q_c < - \delta$ for all $n \geq N_0 + 1$. Let $K\in\R^{1 \times N_0}$ and $L\in\R^{N_0}$ be such that\footnote{Note that $(A_0,\mathfrak{B}_0)$ and $(A_0,C_0)$ satisfy the Kalman condition by arguments from~\cite{lhachemi2021nonlinear}.} $A_0 + \mathfrak{B}_0 K$ and $A_0 - L C_0$ are Hurwitz with eigenvalues that have a real part strictly less than $-\delta<0$. For a given $N \geq N_0 +1$, assume that there exist a symmetric positive definite $P\in\R^{2N \times 2N}$, positive real numbers $\alpha_1,\alpha_2,\alpha_3 > 3/2$ and $\beta,\gamma > 0$ such that 
\begin{equation}\label{eq: thm3 constraints}
\Theta_1 \preceq 0 , \quad
\Theta_2 \leq 0 
\end{equation}
where
\begin{align*}
\Theta_1 & = \begin{bmatrix}
\Theta_{1,1,1} & P \mathcal{L} & P G  \\
\mathcal{L}^\top P & -\beta & 0 \\
G^\top P & 0 & - \alpha_3 \gamma \tilde{\Lambda}^{-1}
\end{bmatrix} 
+ \alpha_2 \gamma \Vert \mathcal{R}_N b \Vert_{L^2}^2 E^\top E \\
\Theta_{1,1,1} & = F^\top P + P F + 2 \delta P + \alpha_1\gamma \Vert \mathcal{R}_N a \Vert_{L^2}^2 \tilde{K}^\top \tilde{K} \\
& \phantom{=}\; + 2\alpha_3\gamma k_f^2 \Vert \mathcal{R}_N b \Vert_{L^2}^2 \tilde{K}^\top \tilde{K} + \alpha_3 \gamma k_f^2 \Omega \\
\Theta_2 & = 2 \gamma \left\{ - \left[ 1 - \frac{1}{2} \left( \frac{1}{\alpha_1} + \frac{1}{\alpha_2} + \frac{1}{\alpha_3} \right) \right] \lambda_{N+1} + q_c + \delta \right\} \\
& \phantom{=}\; + \beta M_\phi + \frac{2 \alpha_3 \gamma k_f^2}{\lambda_{N+1}} 
\end{align*}
with $M_\phi = \sum_{n \geq N+1} \frac{\phi_n(0)^2}{\lambda_n}$. Then, considering the closed-loop system composed of the plant (\ref{eq: RD system 2}) with the system output (\ref{eq: Dirichlet measurement operator}) and the controller (\ref{eq: controller 3}), there exists $M > 0$ such that for any initial conditions $z_0 \in H^2(0,1)$ and $\hat{z}_n(0) \in\R$ so that $\cos(\theta_1) z_0(0) - \sin(\theta_1) z'_0(0) = 0$ and $\cos(\theta_2) z_0(1) + \sin(\theta_2) z'_0(1) = K \hat{Z}^{N_0}(0)$, the system trajectory satisfies 
$$\Vert z(t,\cdot)\Vert_{H^1}^2 + \sum_{n=1}^N \hat{z}_n(t)^2 \leq M e^{-2\kappa t} \left( \Vert z_0 \Vert_{H^1}^2  + \sum_{n=1}^N \hat{z}_n(0)^2 \right)$$ 
for all $t \geq 0$. Moreover, when selecting $N$ to be sufficiently large, there exists $k_f > 0$ (small enough) so that the constraints (\ref{eq: thm3 constraints}) are feasible.
\end{theorem}

\begin{proof}
Consider the Lyapunov functional defined for $X\in\R^{2N}$ and $w \in D(\mathcal{A})$ by 
$$V(X,w) = X^\top P X + \gamma \sum_{n \geq N+1} \lambda_n w_n^2 .$$
The computation of the time derivative of $V$ along the system trajectories (\ref{eq: dynamics w_n}) and (\ref{eq: RD system 3 dynamics truncated model}) reads
\begin{align*}
& \dot{V} + 2 \delta V \\
& = X^\top \{ F^\top P + P F + 2 \delta P \} X + 2 X^\top P \mathcal{L} \zeta + 2 X^\top P G R \\
& \phantom{=}\; + 2\gamma\sum_{n \geq N+1} \lambda_n (-\lambda_n+q_c+\delta) w_n^2 \\
& \phantom{=}\;  + 2\gamma\sum_{n \geq N+1} \lambda_n (a_n u + b_n v + g_n) w_n .
\end{align*} 
With $\tilde{X} = \mathrm{col}(X,\zeta,R)$, noting that $\zeta^2 \leq M_\phi \sum_{n \geq N+1} \lambda_n z_n^2$, and using Young's inequality, (\ref{eq: u in function of X}), and (\ref{eq: expression of v in function of tilde_X}), we infer that
\begin{align*}
& \dot{V} + 2 \delta V \leq \\
& \tilde{X}^\top 
\begin{bmatrix}
F^\top P + P F + 2 \delta P + \alpha_1\gamma \Vert \mathcal{R}_N a \Vert_{L^2}^2 \tilde{K}^\top \tilde{K} & P\mathcal{L} & P G \\
\mathcal{L}^\top P & -\beta & 0 \\
G^\top P & 0 & 0
\end{bmatrix}  
\tilde{X} \\
& + \alpha_2 \gamma \Vert \mathcal{R}_N b \Vert_{L^2}^2 \tilde{X}^\top E^\top E \tilde{X} + \beta M_\phi \sum_{n \geq N+1} \lambda_n w_n^2 \\
& + 2\gamma \sum_{n \geq N+1} \lambda_n \left\{ -\lambda_n + q_c + \delta + \frac{1}{2} \left( \frac{1}{\alpha_1} + \frac{1}{\alpha_2} + \frac{1}{\alpha_3} \right) \lambda_n \right\} w_n^2 \\
& + \alpha_3 \gamma \Vert \mathcal{R}_N g(t,\cdot,z) \Vert_{L^2}^2 .
\end{align*}
for any $\alpha_1,\alpha_2,\alpha_3 > 0$. Using the sector condition (\ref{eq: sector condition bis}) we infer that (\ref{eq: norm residue of semilinearity}) holds. Recalling that $e_n = z_n - \hat{z}_n$, $\tilde{e}_n = \sqrt{\lambda}_n e_n$, and $\tilde{z}_n = \hat{z}_n/\lambda_n$, we deduce that $\sum_{n=1}^{N_0} z_n^2 = \hat{z}_n^2 + 2 \hat{z}_n e_n + e_n^2$ and $\sum_{n=N_{0}+1}^{N} z_n^2 = \lambda_n^2 \tilde{z}_n^2 + 2\sqrt{\lambda_n} \tilde{z}_n \tilde{e}_n + \frac{1}{\lambda_n} \tilde{e}_n^2$, hence $\sum_{n=1}^{N} z_n^2 = X^\top \Omega X$. Moreover, we infer from (\ref{eq: link z_n and w_n}) and using (\ref{eq: u in function of X}) that $\sum_{n \geq N+1} z_n^2 
\leq 2 \sum_{n \geq N+1} w_n^2 + 2 \Vert \mathcal{R}_N b \Vert_{L^2}^2 X^\top \tilde{K}^\top \tilde{K} X$. This implies that 
\begin{equation*}
\dot{V} + 2 \delta V 
\leq \tilde{X}^\top \Theta_1 \tilde{X} + \sum_{n \geq N+1} \lambda_n \Gamma_n z_n^2 
\end{equation*}
with $\Gamma_n = 2\gamma \left\{ - \left[ 1 - \frac{1}{2} \left( \frac{1}{\alpha_1} + \frac{1}{\alpha_2} + \frac{1}{\alpha_3} \right) \right] \lambda_n + q_c + \delta \right\} + \beta M_\phi + \frac{2 \alpha_3 \gamma k_f^2}{\lambda_n} \leq \Theta_2$ for all $n \geq N+1$ where we have used that $1 - \frac{1}{2} \left( \frac{1}{\alpha_1} + \frac{1}{\alpha_2} + \frac{1}{\alpha_3} \right) > 0$ because $\alpha_1,\alpha_2,\alpha_3 > 3/2$. Owing to (\ref{eq: thm3 constraints}), we infer that $\dot{V} + 2 \delta V \leq 0$. Using the definition of $V$ and (\ref{eq: inner product Af and f}), we deduce that the claimed exponential stability estimate holds true.

It remains to show that the constraints (\ref{eq: thm3 constraints}) are feasible when $N$ is large enough and $k_f > 0$ is small enough. Since $\Vert \tilde{C}_1 \Vert = O(1)$ and $\Vert \tilde{\mathfrak{B}}_1 \Vert = O(1)$ as $N \rightarrow + \infty$, the same approach as in the proof of Theorem~\ref{thm2} shows that the unique solution $P \succ 0$ to the Lyapunov equation $F^\top P + P F + 2 \delta P = -I$ is such that $\Vert P \Vert = O(1)$ as $N \rightarrow + \infty$. We fix $\alpha_1,\alpha_2>3/2$ and $\gamma > 0$ arbitrarily. We also fix $\beta=N$, $\alpha_3 = N^3$, and $k_f = 1/N^4$. Hence it can be seen that $\Theta_2 \rightarrow - \infty$ as $N \rightarrow + \infty$. Moreover, since $\Vert \tilde{K} \Vert = \Vert K \Vert$, $\Vert \tilde{L} \Vert = \sqrt{2} \Vert L \Vert$, and $\Vert G \Vert= 1$ are constants independent of $N$, $\Vert P \Vert = O(1)$ and $\Vert E \Vert = O(1)$ as $N \rightarrow + \infty$, and $\Omega \preceq 2 \max(1,1/\lambda_{N_0+1},\lambda_N^2) I$ and $- \tilde{\Lambda}^{-1} \preceq - \min(1,1/\lambda_N) I$, the Schur complement shows that $\Theta_1 \preceq 0$ for $N$ selected to be large enough. This completes the proof.
\end{proof}

\section{Numerical illustration}

For numerical illustration, we consider $p = 1$, $\tilde{q}_0 = -2$, $\tilde{q}_f = 3$, and $\theta_1 = \pi/5$, and $\theta_2 = 0$. Hence the open-loop reaction diffusion PDEs (\ref{eq: RD system 1}) and (\ref{eq: RD system 2}) are unstable in the case of $k_f = 0$, i.e. for $f(t,x,z) = \tilde{q}_f z = 3z$. In the case of the reaction-diffusion PDE with bounded control input (\ref{eq: RD system 1}), we consider the shape function $b(x) = \cos(x) 1\vert_{[1/10,3/10]}$.

\begin{table*}
\center
\begin{tabular}{c||c|c|c|c|c}
Dimension of the observer & $N = 2$ & $N = 3$ & $N = 4$ & $N = 5$ & $N = 6$ \\
\hline\hline
Theorem~\ref{thm2} & $k_f = 1.99$ & $k_f = 2.32$ & $k_f = 2.45$ & $k_f = 2.54$ & $k_f = 2.59$ \\
\hline
Theorem~\ref{thm3} & $k_f = 1.93$ & $k_f = 2.14$ & $k_f = 2.21$ & $k_f = 2.25$ & $k_f = 2.27$
\end{tabular}
\caption{Maximum value of $k_f$ for the sector condition (\ref{eq: sector condition}) obtained for different dimensions $N$ of the observer.}
\label{tab: LMIs}
\end{table*}

The feedback gain $K$ and the observer gain $L$ are computed to achieve the same pole placement in the different studied settings. In the case of Theorem~\ref{thm2}, corresponding to a bounded input operator and a left Dirichlet measurement, the gains are set as $K = -25.9768$ and $L = 5.9341$. In the case of Theorem~\ref{thm3}, corresponding to a right boundary control and a left Dirichlet measurement, the gains are set as $K = -2.3307$ and $L = 5.9341$. 
    
The maximal value of $k_f > 0$ (corresponding to the size of the sector condition (\ref{eq: sector condition}) in which the nonlinearity $f(t,x,z)$ is confined into) for which the stability of the closed-loop system is ensured by applying the theorems of this paper is detailed in Tab.~\ref{tab: LMIs} for the two studied configurations and for different values of the order of the observer.

\section{Conclusion}\label{sec: conclusion}

This paper discussed the topic of output feedback stabilization of semilinear reaction-diffusion PDEs for which the semilinearity is assumed to be confined into a sector condition.  We addressed multiple actuation configurations, including bounded input operators and Robin boundary control. The measurement was selected as a Dirichlet boundary measurement. The adopted control design procedure is systematic in the sense that the derived stability conditions are always feasible provided the order of the controller is selected large enough and the size of the sector condition is small enough. 

It is worth noting that even if the developments of this paper have been focused on a Dirichlet boundary measurement, the case of a Neumann boundary measurement can also be handled using the framework of this paper augmented with the scaling procedure reported in~
\cite{lhachemi2020finite,lhachemi2021nonlinear}.

Future research directions may be concerned with the extension of the method to more stringent nonlinearities such as memory based nonlinearity maps.






\ifCLASSOPTIONcaptionsoff
  \newpage
\fi



\bibliographystyle{IEEEtranS}
\nocite{*}
\bibliography{IEEEabrv,mybibfile}

\begin{thebibliography}{10}
\providecommand{\url}[1]{#1}
\csname url@samestyle\endcsname
\providecommand{\newblock}{\relax}
\providecommand{\bibinfo}[2]{#2}
\providecommand{\BIBentrySTDinterwordspacing}{\spaceskip=0pt\relax}
\providecommand{\BIBentryALTinterwordstretchfactor}{4}
\providecommand{\BIBentryALTinterwordspacing}{\spaceskip=\fontdimen2\font plus
\BIBentryALTinterwordstretchfactor\fontdimen3\font minus
  \fontdimen4\font\relax}
\providecommand{\BIBforeignlanguage}[2]{{%
\expandafter\ifx\csname l@#1\endcsname\relax
\typeout{** WARNING: IEEEtranS.bst: No hyphenation pattern has been}%
\typeout{** loaded for the language `#1'. Using the pattern for}%
\typeout{** the default language instead.}%
\else
\language=\csname l@#1\endcsname
\fi
#2}}
\providecommand{\BIBdecl}{\relax}
\BIBdecl

\bibitem{balas1988finite}
M.~J. Balas, ``Finite-dimensional controllers for linear distributed parameter
  systems: exponential stability using residual mode filters,'' \emph{Journal
  of Mathematical Analysis and Applications}, vol. 133, no.~2, pp. 283--296,
  1988.

\bibitem{boskovic2001boundary}
D.~M. Boskovic, M.~Krstic, and W.~Liu, ``Boundary control of an unstable heat
  equation via measurement of domain-averaged temperature,'' \emph{IEEE
  Transactions on Automatic Control}, vol.~46, no.~12, pp. 2022--2028, 2001.

\bibitem{coron2004global}
J.-M. Coron and E.~Tr{\'e}lat, ``Global steady-state controllability of
  one-dimensional semilinear heat equations,'' \emph{SIAM Journal on Control
  and Optimization}, vol.~43, no.~2, pp. 549--569, 2004.

\bibitem{coron2006global}
------, ``Global steady-state stabilization and controllability of {1D}
  semilinear wave equations,'' \emph{Communications in Contemporary
  Mathematics}, vol.~8, no.~04, pp. 535--567, 2006.

\bibitem{curtain1982finite}
R.~Curtain, ``Finite-dimensional compensator design for parabolic distributed
  systems with point sensors and boundary input,'' \emph{IEEE Transactions on
  Automatic Control}, vol.~27, no.~1, pp. 98--104, 1982.

\bibitem{curtain2012introduction}
R.~F. Curtain and H.~Zwart, \emph{An introduction to infinite-dimensional
  linear systems theory}.\hskip 1em plus 0.5em minus 0.4em\relax Springer
  Science \& Business Media, 2012, vol.~21.

\bibitem{harkort2011finite}
C.~Harkort and J.~Deutscher, ``Finite-dimensional observer-based control of
  linear distributed parameter systems using cascaded output observers,''
  \emph{International journal of control}, vol.~84, no.~1, pp. 107--122, 2011.

\bibitem{karafyllis2021lyapunov}
I.~Karafyllis, ``Lyapunov-based boundary feedback design for parabolic
  {PDEs},'' \emph{International Journal of Control}, vol.~94, no.~5, pp.
  1247--1260, 2021.

\bibitem{katz2020constructive}
R.~Katz and E.~Fridman, ``Constructive method for finite-dimensional
  observer-based control of {1-D} parabolic {PDEs},'' \emph{Automatica}, vol.
  122, p. 109285, 2020.

\bibitem{katzadelayed}
------, ``Delayed finite-dimensional observer-based control of {1-D} parabolic
  {PDEs},'' \emph{Automatica}, vol. 123, p. 109364, 2021.

\bibitem{kitsos2020high}
C.~Kitsos, G.~Besan{\c{c}}on, and C.~Prieur, ``High-gain observer design for
  some semilinear reaction-diffusion systems: a transformation-based
  approach,'' \emph{IEEE Control Systems Letters}, vol.~5, no.~2, pp. 629--634,
  2020.

\bibitem{krstic2008boundary}
M.~Krstic and A.~Smyshlyaev, \emph{Boundary control of PDEs: A course on
  backstepping designs}.\hskip 1em plus 0.5em minus 0.4em\relax SIAM, 2008.

\bibitem{lhachemi2021boundary2}
H.~Lhachemi and C.~Prieur, ``Boundary output feedback stabilization of
  reaction-diffusion {PDEs} with delayed boundary measurement,'' \emph{arXiv
  preprint arXiv:2106.13637}, 2021.

\bibitem{lhachemi2021local}
------, ``Local output feedback stabilization of a reaction-diffusion equation
  with saturated actuation,'' \emph{arXiv preprint arXiv:2103.16523}, 2021.

\bibitem{lhachemi2021nonlinear}
------, ``Nonlinear boundary output feedback stabilization of reaction
  diffusion {PDEs},'' \emph{arXiv preprint arXiv:2105.08418}, 2021.

\bibitem{lhachemi2021predictor}
------, ``Predictor-based output feedback stabilization of an input delayed
  parabolic {PDE} with boundary measurement,'' \emph{Automatica, in press},
  2021.

\bibitem{lhachemi2020finite}
------, ``Finite-dimensional observer-based boundary stabilization of
  reaction-diffusion equations with either a {D}irichlet or {N}eumann boundary
  measurement,'' \emph{Automatica}, vol. 135, p. 109955, 2022.

\bibitem{liu2003boundary}
W.~Liu, ``Boundary feedback stabilization of an unstable heat equation,''
  \emph{SIAM journal on control and optimization}, vol.~42, no.~3, pp.
  1033--1043, 2003.

\bibitem{mazenc2011strict}
F.~Mazenc and C.~Prieur, ``Strict {Lyapunov} functions for semilinear parabolic
  partial differential equations,'' \emph{Mathematical Control and Related
  Fields}, vol.~1, no.~2, pp. 231--250, 2011.

\bibitem{meurer2013extended}
T.~Meurer, ``On the extended {Luenberger-type} observer for semilinear
  distributed-parameter systems,'' \emph{IEEE Transactions on Automatic
  Control}, vol.~58, no.~7, pp. 1732--1743, 2013.

\bibitem{mironchenko2020local}
A.~Mironchenko, C.~Prieur, and F.~Wirth, ``Local stabilization of an unstable
  parabolic equation via saturated controls,'' \emph{IEEE Transactions on
  Automatic Control}, vol.~66, no.~5, pp. 2162--2176, 2021.

\bibitem{orlov2017general}
Y.~Orlov, ``On general properties of eigenvalues and eigenfunctions of a
  {Sturm--Liouville} operator: comments on ''{ISS} with respect to boundary
  disturbances for {1-D} parabolic {PDEs}'','' \emph{IEEE Transactions on
  Automatic Control}, vol.~62, no.~11, pp. 5970--5973, 2017.

\bibitem{prieur2018boundary}
C.~Prieur and J.~J. Winkin, ``Boundary feedback control of linear hyperbolic
  systems: Application to the {Saint-Venant--Exner} equations,''
  \emph{Automatica}, vol.~89, pp. 44--51, 2018.

\bibitem{renardy2006introduction}
M.~Renardy and R.~C. Rogers, \emph{An introduction to partial differential
  equations}.\hskip 1em plus 0.5em minus 0.4em\relax Springer Science \&
  Business Media, 2006, vol.~13.

\bibitem{russell1978controllability}
D.~L. Russell, ``Controllability and stabilizability theory for linear partial
  differential equations: recent progress and open questions,'' \emph{{SIAM}
  Review}, vol.~20, no.~4, pp. 639--739, 1978.

\bibitem{sakawa1983feedback}
Y.~Sakawa, ``Feedback stabilization of linear diffusion systems,'' \emph{SIAM
  Journal on Control and Optimization}, vol.~21, no.~5, pp. 667--676, 1983.

\bibitem{selivanov2016distributed}
A.~Selivanov and E.~Fridman, ``Distributed event-triggered control of diffusion
  semilinear {PDEs},'' \emph{Automatica}, vol.~68, pp. 344--351, 2016.

\bibitem{wu2019improved}
H.-N. Wu, Z.-P. Wang, and H.-X. Li, ``Improved ${H}_\infty$ sampled-data
  control for semilinear parabolic {PDE} systems,'' \emph{International Journal
  of Robust and Nonlinear Control}, vol.~29, no.~6, pp. 1872--1892, 2019.

\bibitem{zhang2020switching}
X.-W. Zhang and H.-N. Wu, ``Switching state observer design for semilinear
  parabolic {PDE} systems with mobile sensors,'' \emph{Journal of the Franklin
  Institute}, vol. 357, no.~2, pp. 1299--1317, 2020.

\end{thebibliography}

\end{document}